\documentclass{amsart}
\usepackage{amsfonts}
\usepackage{mathrsfs}
\usepackage{graphicx}
\usepackage[abs]{overpic}
\usepackage{geometry}
\newtheorem{theorem}{Theorem}[section]
\newtheorem{lemma}[theorem]{Lemma}

\newtheorem{proposition}[theorem]{Proposition}

\theoremstyle{definition}

\newcommand{\R}{\mathbb{R}}

\newcommand{\Z}{\mathbb{Z}}
\newcommand{\tb}{\mathrm{tb}}

\newcommand{\rot}{\mathrm{rot}}

\newcommand{\tr}{\mathrm{tr}}

\theoremstyle{remark}
\newtheorem{remark}[theorem]{Remark}

\numberwithin{equation}{section}

\begin{document}

\title{Strong symplectic fillability of contact torus bundles}

\author{Fan Ding}
\address{School of Mathematical Sciences and LMAM\\ Peking University\\ Beijing
100871, China}
\email{dingfan@math.pku.edu.cn}

\author{Youlin Li}
\address{School of Mathematical Sciences \\Shanghai Jiao Tong University\\
Shanghai 200240, China}
\email{liyoulin@sjtu.edu.cn}

%\subjclass[2010]{}
%\date{\today}

%\keywords{}

\begin{abstract}
In this paper, we study strong symplectic fillability and Stein fillability of some tight contact structures on negative parabolic and negative hyperbolic torus bundles over the
circle. For the universally tight contact structure with twisting $\pi$ in $S^1$-direction on a negative parabolic torus bundle, we completely determine its strong symplectic
fillability and Stein fillability. For the universally tight contact structure with twisting $\pi$ in $S^1$-direction on a negative hyperbolic torus bundle, we give a necessary
condition for it being strongly symplectically fillable. For the virtually overtwisted tight contact structure on the negative parabolic torus bundle with monodromy $-T^n$ ($n<0$), we prove that it is Stein fillable. By the way, we give a partial answer to a conjecture of Golla and Lisca.
\end{abstract}

\maketitle

\section{Introduction}

Tight contact structures on torus bundles are classified up to isotopy, see \cite{G} and \cite {H}.  The study of symplectic fillability and Stein fillability of contact torus bundles has been conducted in the past decade.  Symplectic fillability and Stein fillability of contact elliptic,  positive hyperbolic,  and positive parabolic torus bundles have been completely determined.  See \cite{E}, \cite{DG}, \cite{EtH1}, \cite{LiS}, \cite{Ga3}, \cite{V} and \cite{BO}. In this paper, we focus on strong symplectic fillability and Stein fillability of certain contact negative parabolic and negative hyperbolic torus bundles.

If $A\in SL(2,\Z )$, let $M_A$ denote the $T^2$-bundle over $S^1$ with monodromy $A$. That is, $M_A$ is obtained from $T^2\times I=\R^2/\Z^2\times [0,1]$, with
coordinates $(\mathbf{x},t)=(\left (\begin{array}{l} x \\ y\end{array}\right ) ,t)$ by identifying the two ends via the map $A: T^2\times \{ 1\}\to T^2\times \{ 0\}$, where
$(\mathbf{x},1)\mapsto (A\mathbf{x},0)$. Two torus bundles $M_A$ and $M_B$ are orientation-preserving diffeomorphic  if and only if $A$ is conjugate in $SL(2,\Z )$ to $B$ or $JB^{-1}J^{-1}$,
where $J=\left(
\begin{array}{cc} 0 & 1  \\ 1 & 0
\end{array} \right) $ (cf. \cite[Lemma~6.2]{N}).
If $|\tr(A)|<2$ (resp. $|\tr(A)|=2$ or $|\tr(A)|>2$), then  $A$ and the torus bundle $M_A$ are called  elliptic (resp. parabolic or hyperbolic).  If $\tr (A)>0$ (resp. $\tr(A)<0$), then $A$ and the torus bundle $M_A$ are called  positive (resp.  negative).

Let $\phi:\R\to\R$ be a smooth function with strictly positive derivative. The $1$-form
$$\sin\phi (t)\mathrm{d}x+\cos\phi (t)\mathrm{d}y,\ (x,y,t)\in\R^3,$$
defines a contact structure on $\R^3$. This contact structure descends to a contact structure on $T^2\times\R=(\R^2/\Z^2)\times\R$ which
we denote by $\tilde{\zeta}(\phi)$.

For each $A\in SL(2,\Z)$, $M_A$ is the quotient of $T^2\times\R=(\R^2/\Z^2)\times\R$ with coordinates
$(\mathbf{x},t)=(\left (\begin{array}{l} x \\ y\end{array}\right ) ,t)$ by the transformation $(\mathbf{x},t)\mapsto (A\mathbf{x},t-1)$.

For each $\theta\in\R$, let $\Delta_{\theta}$ denote the ray
$$\left\{ \left(\begin{array}{l} s\cos\theta \\ -s\sin\theta \end{array}\right) :s\ge 0\right\}\subset\R^2.$$
If $A(\Delta_{\phi(t)})=\Delta_{\phi(t-1)}$ for all $t\in\R$, then the contact structure
$\tilde{\zeta}(\phi)$ on $T^2\times\R$ is invariant under the transformation
$(\mathbf{x},t)\mapsto (A\mathbf{x},t-1)$ and thus descends to a contact structure on $M_A$ which we denote
by $\zeta(\phi)$. By \cite[Theorem 1]{DG}, $\zeta(\phi)$ is weakly symplectically fillable.

Let $m$ denote the integer satisfying
$$m\pi\le\sup_{t\in\R}(\phi(t+1)-\phi(t))<(m+1)\pi.$$
Up to fiber preserving isotopy, the contact structure $\zeta(\phi)$ on $M_A$ depends only on $m$ when $m\ge 1$.
This is the universally tight contact structure on $M_A$ with twisting $m\pi$ in $S^1$-direction (see \cite[Theorem 0.1]{H}).
If $A$ is negative parabolic or negative hyperbolic, the set of possible values for $m$ is
the set of positive odd numbers and the contact structure $\zeta(\phi)$ on $M_A$ with the corresponding $m=1$ is denoted by $\xi_A$. The contact structure $\zeta(\phi)$ on $M_A$ with the corresponding $m\ge 3$
has positive Giroux torsion and is not strongly symplectically fillable due to \cite[Corollary 3]{Ga3}.

Let
 $S=\left(
 \begin{array}{cc}   0 & 1 \\  -1 & 0
 \end{array}
\right)  $, and $T=\left(
  \begin{array}{cc}   1 & 1 \\  0 & 1
   \end{array}
\right)$.

 For $n\in\Z$, $M_{-T^n}$, the torus bundle with monodromy $-T^n$, is also denoted by $M_n$. Then $M_{n}, n\in\Z$, constitute all negative parabolic torus bundles.
 The contact structure $\xi_{-T^n}$ on $M_n$ is also denoted by $\xi_n$. When $n\geq -3$, $\xi_n$ is Stein fillable by \cite{V}. In \cite{GoLi1}, Golla and Lisca constructed a
 strongly symplectically  fillable contact structure on $M_n$ for each $n\ge -4$. We first claim that for $-4\le n\le -1$, the strongly symplectically fillable contact structure on $M_n$ they constructed is (contactomorphic to) $\xi_n$ (see Lemma~\ref{utight}).  Thus $\xi_{-4}$ is strongly symplectically fillable.
 In fact, $\xi_{-4}$ is Stein fillable (see Proposition~\ref{Stfillable}). For $n\le -5$, we have the following result.

 \begin{theorem}\label{mainthm}
If $n\leq -5$, then $\xi_{n}$ is not strongly symplectically fillable.
\end{theorem}

\begin{remark}
The negative parabolic torus bundle $M_n$ can be considered as a non-orientable $S^{1}$-bundle over the Klein bottle.  The contact structure $\xi_n$ on $M_n$ is transverse to the $S^{1}$-fibers away from a single torus. By Theorem~\ref{mainthm}, $(M_n,\xi_n)$, $n\le -5$, are examples of contact $3$-manifolds without Giroux torsion that are weakly but
 not strongly symplectically fillable. Niederkr\"{u}ger and Wendl constructed such examples by considering $S^{1}$-invariant contact structures on $S^{1}\times \Sigma$ with $\Sigma$ a closed oriented surface of genus at least $2$ (see \cite[Corollary 5]{NW}).
\end{remark}

\begin{remark}
Let $P_n$ denote the positive parabolic torus bundle with monodromy $T^n$ and $\eta_n$ denote the universally tight contact structure on $P_n$ with twisting $2\pi$ in $S^1$-direction. $P_n$ can be considered as an oriented $S^1$-bundle over the torus with Euler number $n$, and $\eta_n$ is transverse to the $S^1$-fibers away from  two
parallel tori. If $n\ge 0$, then $\eta_n$ is Stein fillable since it can be obtained from $\eta_0$ by Legendrian surgery (see \cite[Proposition 13]{DG} and its
proof). If $n<0$, then $\eta_n$ is not strongly symplectically fillable since it has positive Giroux torsion.
\end{remark} 

According to \cite[Theorem 0.1]{H},  for each $n<0$, there is a unique, up to isotopy, virtually overtwisted tight contact structure on $M_n$. We denote this contact structure on $M_n$ by $\xi_n'$.
We obtain the following.

\begin{proposition}\label{virt.over.}
If $n<0$, then $\xi'_n$ is Stein fillable.
\end{proposition}

Given $(d_1,\ldots,d_k)\in \Z^k,k\ge 1$, we define $$A(d_1,\ldots,d_k):=T^{-d_{k}}S\cdots T^{-d_{1}}S=\left( \begin{array}{ll} d_k & 1 \\ -1 & 0 \end{array}\right)\cdots\left( \begin{array}{ll} d_1 & 1 \\ -1 & 0
\end{array}\right)\in SL(2,\Z ).$$  By \cite[proposition~6.3]{N}, for $A\in SL(2,\Z)$, the torus bundle $M_{A}$ is negative hyperbolic if and only if $A$ is conjugate in $SL(2,\Z)$
to $-A(d_1,\ldots,d_k)$ for some $d_1,\ldots,d_k$ with $d_{i}\ge 2$ for all $i$ and $d_{i}\ge 3$ for some $i$.

Let $$d=(n_1+3,\underbrace{2,\ldots,2}_{m_1},n_2+3,\underbrace{2,\ldots,2}_{m_2},\ldots,n_s+3,\underbrace{2,\ldots,2}_{m_s}),\ \ m_i,n_i\ge 0,s\ge 1,$$ and  $$\rho(d)=(m_{s}+3,\underbrace{2,\ldots,2}_{n_s},m_{s-1}+3,\underbrace{2,\ldots,2}_{n_{s-1}},\ldots,m_{1}+3,\underbrace{2,\ldots,2}_{n_{1}}),$$ then by \cite[Theorem 7.3]{N}  we have $-M_{-A(d)}=M_{-A(\rho(d))}$. If $d$ is embeddable (see
\cite{GoLi1} for the definition), then by \cite[Theorems 1.2 and 2.5]{GoLi1} and \cite{GoLi2}, $\xi_{-A(d)}$ is strongly symplectically fillable. For general $d$, we give a necessary condition for $\xi_{-A(d)}$ to be strongly symplectically fillable.

\begin{theorem}\label{hyperbolic} If $\xi_{-A(d)}$ is strongly symplectically fillable, then $$n_1+n_2+\cdots +n_s\le m_1+m_2+\cdots +m_s+4.$$
\end{theorem}

If $s=1$, then this necessary condition is also sufficient.

\begin{proposition}\label{hyperbolic1} If $d=(n_1+3,\underbrace{2,\ldots,2}_{m_1})$ with $m_1,n_1\ge 0$, then
$\xi_{-A(d)}$ is strongly symplectically fillable if and only if $n_1\le m_1+4$.
\end{proposition}

We also give a partial answer to \cite[Conjecture 1]{GoLi1}.

\begin{proposition}\label{universally tight} Let $(X,\omega)$ be a closed symplectic $4$-manifold obtained as a symplectic blowup of $\mathbb{CP}^2$ with the standard K\"ahler form. Suppose that $$D=C_1\cup\cdots\cup C_l\subset X$$ is a circular, spherical symplectic divisor such that $C_i\cdot C_i\in\{ 0,1\}$ for some $i\in\{ 1,\ldots,l\}$ and the intersection matrix of $D$ is nonsingular. Then, any contact structure induced on the boundary of a concave neighbourhood of $D$ is universally tight.
\end{proposition}

In Section 2, we give some preliminaries. In Section 3, we identify the strongly symplectically fillable contact structure on $M_n$ constructed in \cite{GoLi1} (Lemma~\ref{utight})
and prove Theorems~\ref{mainthm}, \ref{hyperbolic} and Proposition~\ref{hyperbolic1}. In Section 4, we prove Proposition~\ref{virt.over.}. In Section 5, we prove
Proposition~\ref{universally tight}.

\section*{Acknowledgement}
Authors would like to thank John Etnyre and Paolo Lisca for useful email correspondence. The first author is partially supported by
grant no. 11371033 of the National Natural Science Foundation of China. The second author is partially supported by grant no. 11471212 of the
National Natural Science Foundation of China.

\section{Preliminaries}

\subsection{Legendrian surgery on $(M_A,\xi_A)$}
A fiber torus of $M_A$ is a torus $T^2\times\{ p\}\subset M_A$ where $p\in [0,1]$. If $A\in SL(2,\Z)$ is negative
parabolic or negative hyperbolic, then in $(M_A,\xi_A)$, each fiber torus $T^2\times\{ p\}$ is pre-Lagrangian (i.e., linearly foliated)
by the construction of $\xi_A$ (see Section 1).
Using the same method as in the proof of \cite[Proposition 11]{DG}, we can deduce the following proposition. Note that $M_A$ corresponds to $T_{A^{-1}}$ in \cite{DG}.

\begin{proposition}\label{Lsurgery}
 Assume that $A\in SL(2,\Z)$ is negative parabolic or negative hyperbolic. Let $L$ be a simple closed curve on $T^2$ such that $L\times \{ p\}\subset
 T^2\times\{ p\}$ ($p\in [0,1]$) is Legendrian in
 $(M_A,\xi_A)$. If $AT_L$ is negative parabolic or negative hyperbolic, where $T_L\in SL(2,\Z)$ corresponds to a right-handed Dehn twist along $L$ in $T^2$, then the Legendrian surgery along $L\times \{ p\}$
 in $(M_A,\xi_A)$ yields the contact manifold $(M_{AT_L},\xi_{AT_L})$.
\end{proposition}

In $T^2=\R^2/\Z^2$, let $\mu =\{\left( \begin{array}{c} t \\ 0 \end{array}\right) :0\le t\le 1\}$
and $\lambda =\{\left( \begin{array}{c} 0 \\ t \end{array}\right) :0\le t\le 1\}$. Then
$\mu$ (resp. $\lambda$) is a linear simple closed curve in $T^2=\R^2/\Z^2$ of slope $0$ (resp. $\infty$).  Here we use the parameter $t$ to orient $\mu$ and $\lambda$.  The right-handed Dehn twists
$T_{\mu}=\left(
  \begin{array}{cc}   1 & 1 \\  0 & 1
   \end{array}
\right) =T$ and $T_{\lambda}=\left(
  \begin{array}{cc}   1 & 0 \\  -1 & 1
   \end{array}
\right)$.

\subsection{$b_2^+$ and $b_2^-$}

The following lemma is obvious.

\begin{lemma}\label{b2}
Let $X_1,X_2$ be two compact oriented $4$-manifolds. Let $N_i$ be
a component of $\partial X_i$, $i=1,2$. Suppose $f:N_1\to N_2$ is
an orientation-reversing diffeomorphism. The manifold obtained by
gluing $X_1$ and $X_2$ via $f$ is denoted by $X$. Then we have
$b_2^+(X)\ge b_2^+(X_1)+b_2^+(X_2)$ and $b_2^-(X)\ge
b_2^-(X_1)+b_2^-(X_2).$
\end{lemma}

\subsection{Minimal strong symplectic fillings and Stein cobordism}

The following proposition is due to John Etnyre.
\begin{proposition}\cite{Et}\label{minimal}
Let $N'$ be a minimal strong symplectic filling of a contact 3-manifold $(Y_{1}, \xi_{1})$, $W'$ be a Stein cobordism from
$(Y_{1}, \xi_{1})$ to a contact 3-manifold $(Y_{2}, \xi_{2})$. Then $N'\cup W'$ is a minimal strong symplectic filling of $(Y_{2}, \xi_{2})$.
\end{proposition}

\begin{proof}
If there was a symplectic sphere $\Sigma$ of self-intersection $-1$ in $N'\cup W'$,  then according to \cite[Proposition 7.1]{T}, $\Sigma$ is an (almost) complex sphere. It cannot intersect $W'$, otherwise the strictly pluri-subharmonic function would then have a maximum when restricted to the sphere and that can't happen. Thus the sphere $\Sigma$ would have to be entirely contained in $N'$, but that's not possible either since $N'$ was minimal.
\end{proof}

\section{Universally tight  contact torus bundles}

\subsection{Identification of the contact structures constructed on $M_n$ ($-4\le n\le -1$) in \cite{GoLi1}.}

\begin{lemma} \label{utight}
For every integer $-4\le n\le -1$, the strongly symplectically fillable contact structure on $M_n$ constructed in \cite{GoLi1} is (contactomorphic to) $\xi_n$.
\end{lemma}

\begin{proof}
By \cite[Lemma 2.3]{GoLi1}, for each $n\ge -4$ (here $n$ corresponds to $-n$ in \cite[Lemma 2.3]{GoLi1}), there is a spherical complex divisor $D\subset\mathbb{CP}^2\# (5+n)\overline{\mathbb{CP}}^2$ which is the proper transform of a complex line and a smooth conic in general position in $\mathbb{CP}^2$, obtained by blowing up at $4+n$ generic points of the conic and one generic point of the complex line, such that the boundary of a closed regular neighborhood of $D$ is $-M_n$. Since the intersection matrix of $D$ is nonsingular and not negative definite, by \cite[Theorem 1.3]{LM} (see also \cite[Theorem 2.5]{GoLi1}), there is a closed regular neighborhood of $D$ and a symplectic form $\omega$ on $\mathbb{CP}^2\# (5+n)\overline{\mathbb{CP}}^2$ such that $\partial W=-M_n$ is a concave boundary of $(W,\omega)$. The contact structure on $M_n$ constructed in \cite{GoLi1} is induced by the $\omega$-concave structure on $\partial W=-M_n$. We prove the lemma for the case $n=-4$. The proof for other cases are similar.

According to \cite[Theorem 1.1(part B)]{Ga1} and \cite{Ga2},  the contact structure on $M_{-4}$ is supported by an open book decomposition whose page is shown in Figure \ref{fig:19} and whose monodromy is the composition of Dehn twists along the $\pm$-labelled simple closed curves, where the Dehn twists along the $+$(resp. $-$)-labelled curves are right (resp. left) handed. Repeatedly using  \cite[Lemma 4.4.2]{V}, the above open book decomposition is stably equivalent to the open book decomposition shown in Figure \ref{fig:17} with monodromy $\psi_{-4}=t_{\delta_1}t_{\delta_2}t_{\alpha_1}^{-6}t_{\alpha_2}^{-2}$, where $t_{\gamma}$ denotes a right-handed Dehn twist along the simple closed curve $\gamma$.  By part 3(d) of the proof of \cite[Theorem 4.3.1]{V}, the latter open book decomposition supports the contact structure $\xi_{-4}$ on $M_{-4}$. So the lemma holds.
\end{proof}

\begin{figure}[htb]
\begin{overpic}%[grid,tics=10]
{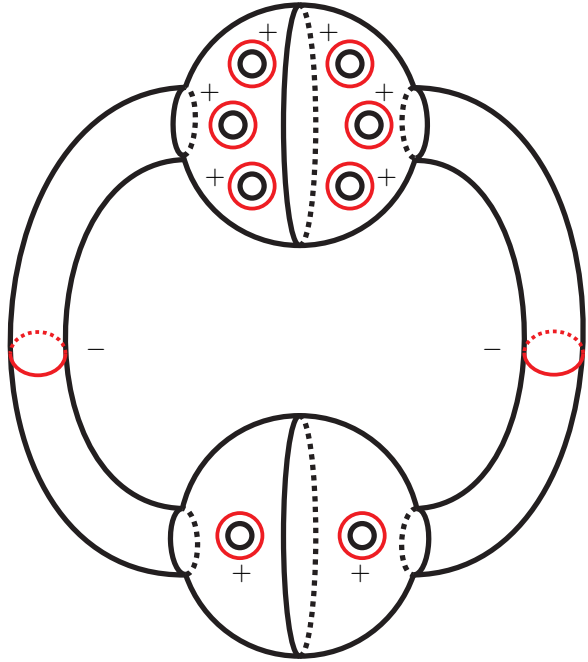}
\put(95, 235){$+$}
\put(73, 212){$+$}
\put(75, 180){$+$}
\put(118, 235){$+$}
\put(139, 212){$+$}
\put(140, 180){$+$}
\put(85, 30){$+$}
\put(130, 30){$+$}
\put(30, 115){$-$}
\put(180, 115){$-$}
%\put(150, 90){$\beta$}

\end{overpic}
\caption{A compact genus one surface with eight boundary components.}
\label{fig:19}
\end{figure}

\begin{figure}[htb]
\begin{overpic}%[grid,tics=10]
{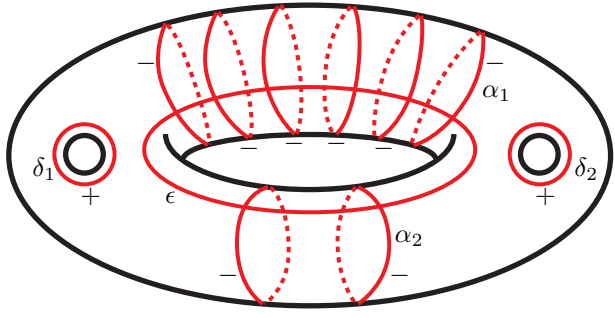}
\put(80, 10){$-$}
\put(145, 10){$-$}
\put(49, 90){$-$}
\put(181, 90){$-$}
\put(88, 58){$-$}
\put(139, 58){$-$}
\put(105, 60){$-$}
\put(121, 60){$-$}
\put(28, 40){$+$}
\put(200, 40){$+$}
\put(180, 80){$\alpha_1$}
\put(147, 25){$\alpha_2$}
\put(10, 50){$\delta_1$}
\put(215, 50){$\delta_2$}
\put(60, 40){$\epsilon$}

\end{overpic}
\caption{A compact genus one surface with two boundary components.}
\label{fig:17}
\end{figure}

\begin{proposition}\label{Stfillable}
For every integer $-4\le n\le -1$, the contact structure $\xi_n$ on $M_n$ is Stein fillable.
\end{proposition}

\begin{proof}
We only prove the proposition for the case $n=-4$. It suffices to show that the monodromy $\psi_{-4}$ admits a  factorization into a product of right-handed Dehn twists.  Using chain relations and braid relations in the mapping class group, we can factor the monodromy $\psi_{-4}$ into a product of right-handed Dehn twists as follows:

\begin{equation*}
\begin{aligned}
\psi_{-4}&=t_{\delta_1}t_{\delta_2}t_{\alpha_1}^{-6}t_{\alpha_2}^{-2}\\
&=t_{\alpha_1}^{-4}t_{\alpha_1}t_{\epsilon}t_{\alpha_2}t_{\alpha_1}t_{\epsilon}t_{\alpha_2}t_{\alpha_1}t_{\epsilon}t_{\alpha_2}t_{\alpha_1}t_{\epsilon}
t_{\alpha_2}t_{\alpha_1}^{-1}t_{\alpha_2}^{-2}t_{\alpha_1}^{-1}\\
&=t_{\alpha_1}^{-3}t_{\epsilon}t_{\alpha_2}t_{\alpha_1}t_{\epsilon}t_{\alpha_2}t_{\alpha_1}t_{\epsilon}t_{\alpha_2}t_{\alpha_1}t_{\epsilon}
t_{\alpha_1}^{-1}t_{\alpha_2}^{-1}t_{\alpha_1}^{-1}\\
&=(t_{\alpha_1}^{-2}t_{\epsilon}t_{\alpha_1}^{2})t_{\alpha_2}t_{\epsilon}(t_{\alpha_1}t_{\alpha_2}t_{\alpha_1}t_{\epsilon}t_{\alpha_1}^{-1}t_{\alpha_2}^{-1}t_{\alpha_1}^{-1}).
\end{aligned}
\end{equation*}
\end{proof}

\subsection{Proof of Theorem~\ref{mainthm}}
The torus bundle $M_{n}$ has monodromy $\left(
  \begin{array}{cc}
    -1 & -n \\   0 & -1 \\
  \end{array}
\right)$. Let $T^2\times\{ p\}$ ($p\in [0,1]$) be a pre-Lagrangian fiber torus in the contact manifold $(M_{n}, \xi_{n})$ such that $\mu\times\{ p\}\subset T^2\times\{ p\}$ is Legendrian (for the definition of $\mu$, see Section 2).
Denote $\mu\times\{ p\}$ by $K$. Note that the contact framing of $K$ coincides with its framing induced by the pre-Lagrangian fiber torus containing it.

 Suppose now that $n\leq-5$. Let $(M_{n}\times [0,1], \omega_n)$ be a symplectization of $(M_n, \xi_n)$. Attaching $(-n-4)$ Weinstein 2-handles to $(M_{n}\times [0,1], \omega_n)$ along $(-n-4)$ parallel copies of $K\times\{ 1\}$ in the fiber torus $T^2\times\{ p\}\times\{ 1\}$ in $M_{n}\times\{1\}$, by \cite[Proposition 2.1]{EtH2}, we obtain a Stein cobordism $W$ such that the concave end is $(M_n, \xi_n)$. By Proposition~\ref{Lsurgery}, the convex end of $W$ is $(M_{-4}, \xi_{-4})$.

\begin{lemma}\label{euler}
$b_{2}^{-}(W)=-n-4$ and $b_{2}^{+}(W)=0$.
\end{lemma}

\begin{proof}
If $n$ is odd, then $H_1(M_n;\Z)\cong\Z\oplus\Z_4$. If $n$ is even, then $H_1(M_n;\Z)\cong\Z\oplus\Z_2\oplus\Z_2$. So $b_1(M_n)=1$, and hence $b_2(M_n)=1$.
Denote a generator of $H_2(M_n\times [0,1]; \R)\cong\R$ by $h_0$.
The inclusion $i:M_{n}\times [0,1]\to W$ induces an injection
$i_{\ast}:H_{2}(M_{n};\R)\rightarrow H_{2}(W;\R)$.
By abuse of notation, $i_{\ast}(h_0)$ is still denoted by $h_0$.
Obviously, for any element $h$ in $H_{2}(W;\R)$, $h_0\cdot h=0$.

Let $[K]\in H_{1}(M_n;\Z)$ denote the homlogy class of $K$. In
$H_1(M_n,\Z)$, we have $2[K]=0$.
There is an oriented annulus in $M_A$ whose boundary consists of two copies of the oriented
$K$. In fact it comes from $\mu\times[0,1]\subset T^2\times [0,1]$ by quotient.
The framing of $K$ induced by the annulus coincides with that induced by
the fiber torus $T^2\times\{ p\}$ containing $K$.

Let $K_1,\cdots,K_{-n-4}$ be the $(-n-4)$ parallel copies of $K\times\{ 1\}$ along which we attach the
$(-n-4)$ Weinstein $2$-handles. For each $i=1,\ldots,-n-4$, there is an oriented annulus $A_i$ in
$M_n\times\{ 1\}$ with $\partial A_i=2K_i$. Let $S_i$ be the oriented surface
which is the union of $A_i$ and two copies of the core disk of the Weinstein $2$-handle attached along $K_i$.
Let $[S_i]\in H_2(W;\R)$ denote the homology class of $S_i$. Then $h_0,[S_1],\ldots,[S_{-n-4}]$ free
generate $H_2(W;\R)$.
Since the $2$-handles are attached to $K_1,\ldots,K_{-n-4}$ with framing $-1$ with respect to the
framing induced by the annuli $A_1,\ldots,A_{-n-4}$,
for $i, j=1,\cdots, -n-4$, we have $$[S_i]\cdot [S_j]=\left\{
\begin{aligned}
-4, &  & i=j, \\
0, &  & i\neq j.
\end{aligned}
\right.$$
Thus $b_{2}^{-}(W)=-n-4$ and $b_{2}^{+}(W)=0$.
\end{proof}

Suppose that $W_n$ is a strong symplectic filling of the contact manifold $(M_n, \xi_n)$.
Without loss of generality, we assume that $W_n$ is minimal.
Then, by Proposition~\ref{minimal}, the union of  $W_{n}$ and $W$ is a minimal strong symplectic filling of $(M_{-4}, \xi_{-4})$. By  \cite[Theorem 3.5]{GoLi1} and its proof, all minimal strong symplectic fillings of the contact manifold $(M_{-4}, \xi_{-4})$ have vanishing $b_2^-$. Indeed,  the union of a minimal strong symplectic filling and the symplectic cap in \cite[Figure 3]{GoLi1} is either $\mathbb{C}P^2\sharp \overline{\mathbb{C}P^2}$ or $S^2\times S^2$.  Since that symplectic cap,  $\mathbb{C}P^2\sharp \overline{\mathbb{C}P^2}$ and $S^2\times S^2$  all have $b_2^-=1$, by Lemma~\ref{b2}, the minimal strong symplectic filling of $(M_{-4}, \xi_{-4})$ has $b_2^-=0$.  Hence $b_2^-(W_{n}\cup W)=0$.   So by Lemma~\ref{b2}, $b_2^-(W)=0$.  This contradicts Lemma~\ref{euler} since $n\le -5$. Thus $(M_n, \xi_n)$ is not strongly symplectically fillable for $n\le -5$.

\subsection{Proof of Theorem \ref{hyperbolic}}

Let $A=\left( \begin{array}{ll} x & y \\ z & w
\end{array}\right)\in SL(2, \mathbb{Z})$. Assume that $\tr (A)=x+w\le -3$, i.e., $A$ is negative hyperbolic.
Let $T^2\times\{ p\}$ ($p\in [0,1]$) be a pre-Lagrangian fiber torus in the contact manifold $(M_A,\xi_A)$ such that
$\lambda\times \{ p\}\subset T^2\times\{ p\}$ is Legendrian (for the definition of $\lambda$, see Section 2).
Denote $\lambda\times\{ p\}$ by $K$. Let $(M_A\times [0,1],\omega_A)$ be
a symplectization of $(M_A,\xi_A)$. Attaching a Weinstein 2-handle
to $(M_A\times [0,1], \omega_A)$ along $K\times\{ 1\}$ in the
fiber torus $T^2\times\{ p\}\times \{ 1\}$ in $M_A\times\{1\}$, by
\cite[Proposition 2.1]{EtH2}, we obtain a Stein cobordism $W'$
such that the concave end is $(M_A,\xi_A)$. Let $A'=\left(
\begin{array}{ll} x & y \\ z & w
\end{array}\right)\left(
\begin{array}{ll} 1 & 0 \\ -1 & 1
\end{array}\right)=\left(
\begin{array}{ll} x-y & y \\ z-w & w
\end{array}\right)$. Suppose that $\tr (A')=x+w-y\le -3$. Then
by Proposition~\ref{Lsurgery}, the convex end of $W'$ is $(M_{A'},\xi_{A'})$.
Denote $\mu\times\{ 0\}$ ($\subset T^2\times\{ 0\}\subset M_A$) by $\mu_0$ and
$\lambda\times\{ 0\}$ ($\subset T^2\times\{ 0\}\subset M_A$) by $\lambda_0$.
Denote $\mu\times\{ 1\}$ ($\subset T^2\times\{ 1\}\subset M_A$) by $\mu_1$ and
$\lambda\times\{ 1\}$ ($\subset T^2\times\{ 1\}\subset M_A$) by $\lambda_1$.

\begin{lemma}\label{b2+-}
Assume that $\mathrm{tr}(A')\le -3$, then $b_2^+(W')=0$ and
$b_2^-(W')=1$.
\end{lemma}

\begin{proof} Since $b_1(M_A)=1$, $b_2(M_A)=1$. Denote a generator of $H_{2}(M_A\times [0,1];
\R )\cong\R $ by $h_0$. The inclusion $i:M_A\times [0,1]\rightarrow W'$ induces an injection $i_{\ast}:H_{2}(M_A;\R )\rightarrow
H_{2}(W';\R )$. By abuse of notation, $i_{\ast}(h_0)$ is still denoted by $h_0$. Obviously, for any element $h$ in
$H_{2}(W';\R )$, $h_0\cdot h=0$.

Let $[\mu_0],[\lambda_0],[K]\in H_1(M_A;\Z)$ denote the homology classes of $\mu_0,\lambda_0,K$. In $H_{1}(M_A;\Z )$, we have $[\mu_0]=x[\mu_0]+z[\lambda_0],[\lambda_0]=y[\mu_0]+w[\lambda_0]$. Thus $(2-x-w)[\lambda_0]=0$.
Since $[K]=[\lambda_0]$, $(2-x-w)[K]=0$. Let $C$ be a $2$-chain in $M_A$ with $\partial C=(2-x-w)K$. Let $S$ be the oriented surface which is
the union of $C\times\{ 1\}$ ($\subset M_A\times\{ 1\}$) and $2-x-w$ copies of the core disk of the attached Weinstein $2$-handle.
Let $[S]\in H_2(W';\R)$ denote the homology class of $S$. Then $h_0,[S]$ free generate $H_2(W';\R)$. By  Lemma~\ref{selfinter} below, $[S]\cdot[S]=-(2-x-w)(2-x-w+y)=-(2-\tr (A))(2-\tr (A'))<0$. Therefore, $b_2^+(W')=0$ and $b_2^-(W')=1$.
\end{proof}

\begin{lemma}\label{selfinter}
$[S]\cdot [S]=-(2-x-w)(2-x-w+y)$.
\end{lemma}

\begin{proof} Without loss of generality, we assume that $0<p<1$ . Since $(2-x-w)[K]=0$ in $H_1(M_A;\Z)$, $K$ is a rationally null-homologous knot in $M_A$. Denote a
closed regular neighborhood of $K$ in $M_A$ by $\nu (K)$.  Let $\lambda'\subset\partial \nu (K)$ be a
longitude for $K$ determined by the framing induced by the fiber torus $T^2\times\{ p\}$ containing $K$. Let $\mu'\subset\partial\nu(K)$ be a meridian of $\nu(K)$ oriented such that the intersection number $\mu'\cdot\lambda'=1$ on $\partial\nu(K)$.
In $H_1(M_A\setminus\mathrm{Int}(\nu(K));\Z)$, we have
$[\mu_1]=x[\mu_0]+z[\lambda_0],[\lambda_1]=y[\mu_0]+w[\lambda_0],[\lambda_1]=[\lambda_0]=[\lambda']$ and $[\mu']=[\mu_1]-[\mu_0]$.
So in $H_1(M_A\setminus\mathrm{Int}(\nu(K));\Z)$, $(2-x-w)[\lambda']+y[\mu']=0$. Since the $2$-handle is attached to $K\times\{ 1\}$ with framing
$-1$ with respect to the framing induced by the fiber torus $T^2\times\{ p\}\times\{ 1\}$ containing it,  the Lemma follows from \cite[Lemma 5.1]{MT}.
\end{proof}

Let
$$d=(n_1+3,\underbrace{2,\ldots,2}_{m_1}, n_2+3,\underbrace{2,\ldots,2}_{m_2},\ldots, n_s+3,\underbrace{2,\ldots,2}_{m_s}),\ \ m_i,n_i\ge 0,s\ge 1.$$  Suppose that $s\ge 2$. Let $$d'=(\underbrace{2,\ldots,2}_{m_{1}+1}, n_2+3,\underbrace{2,\ldots,2}_{m_2},\ldots, n_s+3,\underbrace{2,\ldots,2}_{m_s}).$$
Suppose that $T^2\times\{ p\}$ ($p\in [0,1]$) is a pre-Lagrangian fiber torus in the contact manifold $(M_{-A(d)},\xi_{-A(d)})$ such that
$\lambda\times \{ p\}\subset T^2\times\{ p\}$ is Legendrian.
Let $(M_{-A(d)}\times [0,1], \omega_{-A(d)})$ be a symplectization of $(M_{-A(d)},\xi_{-A(d)})$.  As before, attaching $n_{1}+1$ Weinstein 2-handles (if $s=1$, attaching $n_1$ Weinstein $2$-handles) to $(M_{-A(d)}\times [0,1], \omega_{-A(d)})$ along parallel copies of $\lambda\times\{ p\}\times\{ 1\}$ in the fiber torus
$T^2\times\{ p\}\times\{ 1\}$ in $M_{-A(d)}\times\{1\}$, we obtain a Stein cobordism such that the concave end is $(M_{-A(d)},\xi_{-A(d)})$ and the convex end is $(M_{-A(d')},\xi_{-A(d')})$ by Proposition~\ref{Lsurgery}.   Let $$d''=(n_2+3,\underbrace{2,\ldots,2}_{m_2},\ldots, n_s +3,\underbrace{2,\ldots,2}_{m_{s}+m_{1}+1}),$$ then $(M_{-A(d')},\xi_{-A(d')})=(M_{-A(d'')},\xi_{-A(d'')})$. So we can attach Weinstein $2$-handles as before.
After successively attaching $n_1+n_2+\cdots +n_s+s-1$ Weinstein 2-handles, we obtain a Stein cobordism $W$ with the concave end $(M_{-A(d)},\xi_{-A(d)})$ and the convex end $(M_{-A(d_0)},\xi_{-A(d_0)})$, where $$d_0=(3,\underbrace{2,\ldots,2}_{m_1+m_2+\cdots +m_s+s-1}).$$  By Lemma~\ref{b2+-} and Lemma~\ref{b2}, we have $b_2^-(W)\ge n_1+n_2+\cdots +n_s+s-1$.

Suppose that $W_{-A(d)}$ is a minimal strong symplectic filling of the contact manifold $(M_{-A(d)},\xi_{-A(d)})$. Then by
Proposition~\ref{minimal},  $W_{-A(d)}\cup W$ is a minimal strong symplectic filling of $(M_{-A(d_0)},$
$\xi_{-A(d_0)})$. Let $c=m_1+m_2+\cdots +m_s+s+2$. Then $M_{-A(d_0)}=-M_{-A(c)}$. By \cite[Theorem 3.1]{GoLi1}, the contact manifold $(M_{-A(d_0)}, \xi_{-A(d_0)})$ admits a unique minimal strong symplectic filling up to orientation preserving diffeomorphism, which is the complement of the interior of a closed regular neighborhood $\tilde{W}$ of a spherical complex divisor $D$ in $\mathbb{C}P^2\#
(c+2)\overline{\mathbb{C}P^{2}}$. By the construction of $D$ (see \cite[Lemma 2.4]{GoLi1}), $b_2^-(\tilde{W})=1$.
Hence by Lemma~\ref{b2}, $b_2^-(W)\le b_2^-(W_{-A(d)}\cup W)\le b_2^-(\mathbb{C}P^2\#
(c+2)\overline{\mathbb{C}P^{2}})-b_2^-(\tilde{W})=c+2-1=c+1$. So $n_1+n_2+\cdots +n_s+s-1\le m_1+m_2+\cdots
+m_s+s+2+1$, i.e., $n_1+n_2+\cdots +n_s\le m_1+m_2+\cdots +m_s+4$,
concluding the proof of Theorem \ref{hyperbolic}.

\subsection{Proof of Proposition \ref{hyperbolic1}}
It suffices to show that if $n_1\le m_1+4$, then $\xi_{-A(d)}$ is strongly symplectically fillable.

If $2\le n_{1}\le m_{1}+4$, then  $(n_{1}-1, 1, \underbrace{2,\ldots,2}_{n_{1}-2}, 1)\prec \rho(d)=(m_{1}+3,\underbrace{2,\ldots,2}_{n_{1}})$. Since $(n_{1}-1, 1, \underbrace{2,\ldots,2}_{n_{1}-2}, 1)$ is a blowup of $(0,0)$, $d$ is embeddable. We refer the reader to \cite{GoLi1}  for the notation ``$\prec$" and the definition of ``blowup".  If $n_{1}=1$, then $(0, 0)\prec\rho(d)=(m_{1}+3, 2)$. So $d$ is also embeddable. By \cite[Theorem 1.2 and Theorem 2.5(v)]{GoLi1},  for $n_1\ge 1$, $\xi_{-A(d)}$ is strongly symplectically fillable.  If $n_{1}=0$, then $\rho(d)=(m_{1}+3)$. By \cite[Theorem 1.2 and Theorem 2.5(iv)]{GoLi1} (\cite[Theorem 1.2]{GoLi1} is still true if
$d=(3,\underbrace{2,\ldots,2}_{m_1})$),  $\xi_{-A(d)}$ is strongly symplectically fillable.

\section{Virtually overtwisted contact torus bundles}

In this section, we prove Proposition~\ref{virt.over.}.
Since $-T^{n}=(-S)^{2}T^{n}$ is conjugate to $(-S)T^{n}(-S)$, by \cite[Theorem A.4]{KM}, $M_n$ can be obtained by surgery of $S^3$
 along the framed link shown in the left of  Figure~\ref{fig:kdiagram}. We isotope the framed link to the right of Figure~\ref{fig:kdiagram}. So we can turn the topological surgery diagram in Figure~\ref{fig:kdiagram} to a  Legendrian link diagram in standard form (cf. \cite{Gom} or \cite{GomS}) shown in Figure~\ref{fig:lhddiagram}.   The Legendrian knot $K_0$ in Figure~\ref{fig:lhddiagram} has $\tb (K_0)=1$.  Performing $-n$ positive or negative stabilizations on $K_0$, we obtain a Legendrian knot $K'_{0}$ which has $\tb (K'_{0})=n+1$. Attaching a Weinstein $2$-handle along $K'_{0}$ yields a Stein domain $N$ which is a Stein filling of a contact structure $\xi$ on $M_{n}$.
 It is easy to know that  $H_1(N; \Z )\cong\Z\oplus\Z_2$.

\begin{figure}[htb]
\begin{overpic}%[grid,tics=10]
{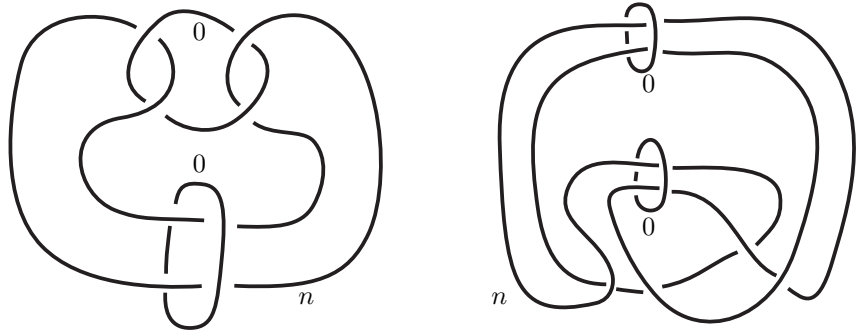}
\put(70, 60){$0$}
\put(70, 110){$0$}
\put(183, 10){$n$}
\put(240, 36){$0$}
\put(240, 90){$0$}
\put(110, 10){$n$}

\end{overpic}
\caption{An isotopy of topological surgery diagrams.}
\label{fig:kdiagram}
\end{figure}

\begin{figure}[htb]
\begin{overpic}%[grid,tics=10]
{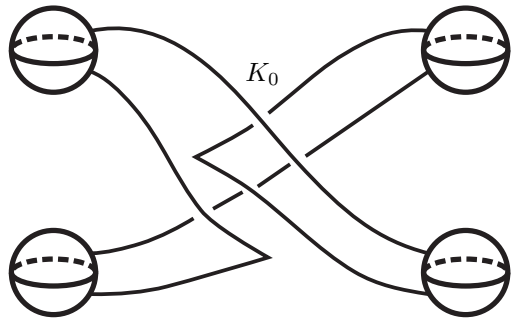}
\put(90, 90){$K_{0}$}
%\put(370, -5){$L_{1}$}
%\put(220, -5){$L_{-1}$}
\end{overpic}
\caption{A Legendrian link diagram in standard form, where $K_0$ is a Legendrian knot with $\tb (K_0)=1$, $\rot (K_0)=0$.}
\label{fig:lhddiagram}
\end{figure}

Let $\widetilde{M_n}$ denote the double cover of $M_n$ corresponding to the epimorphism $$\phi:\pi_1(M_n)\stackrel{pr}{\to}\pi_1(S^1)\cong\Z\stackrel{\beta}{\to}\Z_2,$$  where $pr$ is induced by the projection of the $T^2$-bundle and $\beta:\Z\to\Z_2$ denotes the homomorphism which sends $1$ to the generator of $\Z_2$. Then $\widetilde{M_n}$ is a $T^2$-bundle over $S^1$ with monodromy $T^{2n}$.

\begin{lemma} \label{doublecover1}
If $\xi$ is a universally tight contact structure on $M_n$, then the lift of  $\xi$ to $\widetilde{M_n}$ is not strongly symplectically fillable.
\end{lemma}

\begin{proof}
According to Honda's classification \cite[Theorem 0.1]{H}, we divide the proof into two cases.
If $\xi$ is a universally tight contact structure on $M_n$ with twisting in $S^{1}$-direction $\beta_{S^{1}}\ge \pi$, then the lift of $\xi$ to $\widetilde{M_n}$ is universally tight with $\beta_{S^{1}}\ge 2\pi$.   Explicitly, the lift of $\xi$  can be written as given by the following $1$-form on $(T^{2}\times \mathbb{R})/\sim$: $\alpha_{m}=\sin(\phi(t))dx+\cos(\phi(t))dy,$
with $m\in \mathbb{Z}^{+}$,  $\phi'(t)>0$, $2m\pi\leq \sup_{t\in R}(\phi(t+1)-\phi(t))<(2m+1)\pi$, and $\ker\alpha_m$ is invariant
under the action $(\mathbf{x}, t)\rightarrow (T^{2n}\mathbf{x}, t-1)$. See the second paragraph in page 99 of \cite{H} or Section 1.  If $\xi$ is a universally tight contact structure on $M_n$ with minimal twisting in the $S^1$-direction given by $\mu'$ if $\mu'$ is odd, or by $(\mu', \pm)$ if $\mu'$ is even, where $\mu'$ is a positive integer, then the lift of $\xi$ to $\widetilde{M_n}$ is a universally tight contact structure on $\widetilde{M_n}$ with minimal twisting in the $S^1$-direction given by $(\mu', l')$ for some integer $l'$, which is contactomorphic to a universally tight contact structure on $\widetilde{M_n}$ with $\beta_{S^{1}}\ge 2\pi$.

Since $n<0$,  it is straightforward to check that a universally tight contact structure on $\widetilde{M_n}$ with
$\beta_{S^1}\ge 2\pi$ has positive Giroux torsion.  So the lemma follows from \cite[Corollary 3]{Ga3}.
\end{proof}

Let $c:H_1(N;\Z)\to\Z$ denote a homomorphism which sends a generator of the $\Z$-summand of $H_1(N;\Z)$ to a
generator of $\Z$. Let $\tilde{N}$ denote the double cover of $N$ corresponding to the epimorphism $$\psi:\pi_1(N)\stackrel{h}{\to} H_1(N;\Z)\stackrel{c}{\to}\Z\stackrel{\beta}{\to}\Z_2,$$ where $h$ denotes the Hurewicz homomorphism. Since $N$ is Stein, the homomorphism $j:\pi_1(M_n)\to\pi_1(N)$ induced by inclusion is surjective. Thus the boundary of $\tilde{N}$, $\partial\tilde{N}$, is a double cover of $M_n$ corresponding to the epimorphism $\psi\circ j:\pi_1(M_n)\to\Z_2$.

\begin{lemma} \label{doublecover2}
$\psi\circ j=\phi$.
\end{lemma}

\begin{proof}
Note that $h\circ j=j_0\circ h_0$, where $h_0:\pi_1(M_n)\to H_1(M_n;\Z)$ is the Hurewicz homomorphism and $j_0:H_1(M_n;\Z)\to H_1(N;\Z)$ is induced by inclusion. Thus $\psi\circ j=\beta\circ c\circ h\circ j=\beta\circ c\circ j_0\circ h_0$. If $n$ is odd, then $H_1(M_n;\Z)\cong\Z\oplus\Z_4$. If $n$ is even, then $H_1(M_n;\Z)\cong\Z\oplus\Z_2\oplus\Z_2$. The homomorphism $c\circ j_0:H_1(M_n;\Z)\to\Z$
sends a torsion element of $H_1(M_n;\Z)$ to $0$. Since $c\circ j_0$ is surjective, it sends a generator of the $\Z$-summand of $H_1(M_n;\Z)$ to a generator of $\Z$.  Note that $\phi=\beta\circ pr=\beta\circ p_0\circ h_0$, where $p_0:H_1(M_n;\Z)\to H_1(S^1;\Z)\cong\Z$ is induced by the projection of the $T^2$-bundle. $p_0$ sends a torsion element of $H_1(M_n;\Z)$ to $0$.
Since $p_0$ is surjective, it sends a generator of the $\Z$-summand of $H_1(M_n;\Z)$ to a generator of $\Z$. Thus $\beta\circ p_0=\beta\circ c\circ j_0$.
Hence $\psi\circ j=\beta\circ c\circ j_0\circ h_0=\beta\circ p_0\circ h_0=\phi$.
\end{proof}

By Lemma~\ref{doublecover2}, $\partial\tilde{N}=\widetilde{M_n}$. Lift the contact structure $\xi$ to $\partial\tilde{N}$ and denote the resulting contact structure by $\tilde{\xi}$. Since the Stein structure on $N$ lifts to $\tilde{N}$, $\tilde{\xi}$ is  a Stein fillable contact structure on $\widetilde{M_n}$. By Lemma~\ref{doublecover1}, $\xi$ is  not universally tight. It follows that $\xi$ is just the virtually overtwisted contact structure $\xi_n'$ and $\xi_n'$ is Stein fillable.

\section{Circular spherical symplectic divisors}

Let $(X,\omega)$ be a closed symplectic $4$-manifold obtained as a symplectic blowup of $\mathbb{CP}^2$ with the standard K\"ahler form. For a circular, spherical symplectic divisor $D=C_1\cup\cdots\cup
C_l\subset X$ ($l\ge 2$), define $e_i=C_i\cdot C_i,i=1,\ldots,l$. The boundary of a closed regular neighborhood of $D$ is $M_A$, a torus bundle over $S^1$, with $A=A(-e_1,\ldots,-e_l)$ (see the proof of \cite[Theorem 6.1]{N}).

Now we start the proof of Proposition~\ref{universally tight}. Assume that $e_i\in\{ 0,1\}$ for some $i\in\{ 1,\ldots,l\}$. Then the intersection matrix of $D$ is not negative definite. Since the intersection matrix of $D$ is nonsingular, we can apply \cite[Theorem 2.1]{GoLi1} to see that there is a closed regular neighborhood $V$ of $D$ and
a deformation $\omega_1$ of $\omega$ such that $\partial V$ is a concave boundary of $(V,\omega_1)$. Denote the induced contact structure on $\partial V=M_A$ by $\xi$. The contact manifold $(-M_A,\xi )$ admits a strong symplectic filling $P_0$ given by the complement of $\mathrm{Int}(V)$.

The intersection matrix of $D$ is nonsingular implies that $b_1(M_A)=1$ (see the proof of \cite[Theorem 2.5]{GoLi1}). Hence
$\mathrm{tr}(A)\neq 2$.

\begin{lemma} \label{no filling}
$(-M_A,\xi )$ cannot admit a strong symplectic filling $P$ with $b_1(P)=1$.
\end{lemma}

\begin{proof} Suppose that $(-M_A,\xi )$ admits a strong symplectic filling $P$ with $b_1(P)=1$. By \cite[Theorems 1.1 and 1.4]{Mc}, $V\cup P$ is rational or ruled. So $b_1(V\cup P)$ is even. By the Mayer-Vietoris sequence of $(V,P)$, the fact that $i_{\ast}:H_1(M_A;\Z)\to H_1 (V;\Z)\cong\Z$ induced by inclusion is surjective and $b_1(V)=b_1(P)=1$, we conclude that $b_1(V\cup P)=1$, a
contradiction.
\end{proof}

We prove Proposition~\ref{universally tight} by a case by case argument.

If $A$ is elliptic and $\xi$ is not universally tight, then $\xi$ is  one of the three virtually overtwisted contact structures listed in elliptic
case of \cite[Theorem 0.1]{H} (one for $A^{-1}$ conjugate to $S$, two for $A^{-1}$ conjugate to $(T^{-1}S)^2$). By
\cite[Theorem 1.1]{EtH1}, these three contact structures are not weakly symplectically semi-fillable, contradicting the fact that $(-M_A,\xi )$ admits the strong symplectic filling $P_0$.

If $A$ is hyperbolic with $\mathrm{tr}(A)>2$, since $(-M_A,\xi)$ is strongly symplectically fillable, it is a tight contact structure which is minimally twisting in
the $S^1$-direction. Therefore there is a Stein filling $P$ of $(-M_A,\xi)$ with $b_1(P)=1$(see the proof of \cite[Proposition 11]{BO}). This contradicts Lemma~\ref{no filling}.

If $A$ is hyperbolic with $\mathrm{tr}(A)<-2$ and $(-M_A,\xi)$ is not universally tight, then it is a tight contact structure which is minimally twisting in the
$S^1$-dirction. By \cite[Lemma 4.3]{GoLi1}, there is a Stein filling $P$ of $(-M_A,\xi)$ with $b_1(P)=1$, contradicting
Lemma~\ref{no filling}.

If $A$ is parabolic with $\mathrm{tr}(A)=-2$ and $(-M_A,\xi)$ is not universally tight, then it is the virtually overtwisted
contact structure in the preceding section by Honda's classification \cite{H}. By  the preceding section, it has a Stein filling $P$ with $b_1(P)=1$, contradicting Lemma~\ref{no filling}.

This ends the proof of Proposition~\ref{universally tight}.

\end{document}